\documentclass[12pt,a4paper]{article}
\usepackage[centertags]{amsmath}
\usepackage{amsfonts}

\usepackage{graphics}
\usepackage{amssymb}
\usepackage{amsthm}
\usepackage{newlfont}
\usepackage{epsfig}
\usepackage{amssymb}
\usepackage{amsmath}
\usepackage{amsthm}
\usepackage{graphicx}
\usepackage{makeidx}
 \usepackage{lineno}
\theoremstyle{plain}
\newtheorem{thm}{Theorem}[section]
\newtheorem{cor}[thm]{Corollary}

\newtheorem{prop}[thm]{Proposition}
\theoremstyle{definition}

\theoremstyle{remark} \tolerance=10000 \hbadness=10000
\def \ni{\noindent}
\author{
    Anu V.\footnote{E-mail : anusaji1980@gmail.com}\\ Department of Mathematics\\
    St.Peter's College\\ Kolenchery - 682 311\\  Kerala, India.\and
    Aparna Lakshmanan S.\footnote{E-mail : aparnaren@gmail.com}\\
    Department of Mathematics\\St.Xavier's College for Women\\Aluva -
    683 101\\\vspace{0.2cm} Kerala, India.}

\title{Impact of Some Graph Operations on Double Roman Domination Number}

\date{}
\begin{document}

\maketitle

\begin{abstract}

Given a graph $G=(V,E)$,  a function $f:V\rightarrow \{0,1,2,3\}$
having the property that if $f(v)=0$, then there exist $
v_{1},v_{2}\in N(v)$ such that $f(v_{1})=f(v_{2})=2$ or there
exists $ w \in N(v)$ such that $f(w)=3$, and if $f(v)=1$, then
there exists $ w \in N(v)$ such that $f(w)\geq 2$ is called a
double Roman dominating function (DRDF). The weight of a DRDF $f$ is
the sum $f(V)=\sum_{v\in V}f(v)$. The double Roman domination
number, $\gamma_{dR}(G)$, is the minimum among the weights of  DRDFs on $G$. In
this paper, we   study the impact of  some graph operations, such as cartesian product,  addition of twins and corona with a graph,  on
double Roman domination number.\\

\ni {\bf Keywords:} Domination Number,  Roman Domination Number,  Double Roman Domination Number,  Cartesian Product.\\

%\PACS{PACS code1 \and PACS code2 \and more}

\end{abstract}

%\nin {\bf AMS Subject Classification:} {\bf 05C75} (primary),
 % {\bf 05C62, 68Q17} (secondary).

\section{Introduction}
\label{intro}
Let $G = (V(G),E(G))$ be a graph with vertex set $V(G)$ and edge
set $E(G)$. If there is no ambiguity in the choice of $G$, then we
write $V(G)$ and $E(G)$ as $V$ and $E$ respectively.  Let
$f:V\rightarrow \{0,1,2,3\}$ be a function defined on $V(G). $  Let $V_{i}^{f}=\{v\in V(G) : f(v)=i\}.$ (If there is no ambiguity, $V_{i}^{f}$ is written as $V_{i}$.) Then $f$ is a double Roman dominating function
(DRDF)
on  $G$ if it satisfies the  following conditions.\\
(i) If $v\in V_{0}$, then vertex $v$ must have at least two
neighbors
in $V_{2}$ or at least one neighbor in $V_{3}$.\\
(ii) If $v\in V_{1}$, then vertex $v$ must have at least one
neighbor
in $V_{2} \cup V_{3}$.\par

The weight of a DRDF $f$ is
the sum $f(V)=\sum_{v\in V}f(v)$. The double Roman domination
number, $\gamma_{dR}(G)$, is the minimum among the weights of  DRDFs on $G$, and a DRDF
on $G$ with weight $\gamma_{dR}(G)$ is called a
$\gamma_{dR}$-function of $G$
\cite{Bee}.\par

The  study of double Roman domination was  initiated by R. A. Beeler, T. W.
Haynes and S. T. Hedetniemi in \cite{Bee}.
They  studied
the relationship between double Roman domination and Roman
domination and the bounds on the double Roman domination number of
a graph $G$ in terms of its domination number. They also
determined a sharp upper bound on $\gamma_{dR}(G)$ in terms of the
order of $G$ and characterized the graphs attaining this bound.  In \cite{Abd1}, it is proved that the decision problem associated with  $\gamma_{dR}(G)$ is NP-complete for bipartite and chordal graphs. Moreover, a characterization of graphs $G$ with small $\gamma_{dR}(G)$ is provided. In \cite{Hao}, G. Hao et al. initiated the study of the double Roman domination of digraphs. L. Volkmann gave a sharp lower bound on $\gamma_{dR}(G)$ in \cite{Vol}. In \cite{Anu}, it is proved that  $\gamma_{dR}(G)+2 \leqslant \gamma_{dR}(M(G)) \leqslant \gamma_{dR}(G)+3$,
where $M(G)$ is the Mycielskian graph of $G$. It is also proved  that there
is no relation between the double Roman domination number of a
graph and its induced subgraphs.  In \cite{Amj}, J. Amjadi et al.  improved an upper bound on  $\gamma_{dR}(G)$ given in \cite{Bee} by showing that for any connected graph $G$ of order $n$ with minimum degree at least two, $\gamma_{dR}(G)\leqslant \frac{8n}{7}$.

\subsection{Basic Definitions and Preliminaries}
The open
neighborhood of  a vertex $v\in V$ is the set $N(v)=\{u: uv \in
E\}$, and its closed neighborhood is $N[v]=N(v)\cup \{v\}$.
The vertices in $ N(v)$ are called the neighbors of $v$. For a set $D \subseteq V $, the open neighborhood is $N(D)=\cup_{v\in D}N(v)$ and the closed neighborhood is $N[D]=N(D)\cup D$. A set $D$ is a dominating set if $N[D]=V$. The domination number $\gamma(G)$ is the minimum cardinality of a dominating set in $G$.
A false twin of a vertex $u$ is a vertex $u'$ which is adjacent
to all the vertices in $N(u)$.  A true twin of a vertex $u$ is a vertex $u'$ which is adjacent
to all the vertices in $N[u]$. A subset $S$ of the vertex set $V$ of a graph $G$ is called independent if no two vertices of $S$ are adjacent in $G$.\par

If $f:A\rightarrow B$ is a function from $A$ to $B$, and $C$ is a subset of $A$, then the restriction of $f$ to $C$ is the function which is defined by the same rule as $f$ but with a smaller domain set $C$ and is denoted by $f|_{C}$.\par

A complete graph on $n$ vertices, denoted by  $K_{n}$, is the graph in which any two vertices are adjacent. A trivial graph is a graph with no edges. A path on $n$ vertices $P_{n}$  is the
graph with vertex set $\left\{v_{1},v_{2},\dots,v_{n} \right\}$
and $v_{i}$ is adjacent to
$v_{i+1}$ for $i=1,2,\dots,n-1$. If in addition, $v_{n}$ is adjacent
to $v_{1}$ and $n\geq 3$, it is called  a cycle of length
$n$, denoted by $C_{n}$. A universal vertex is a vertex adjacent to all the other vertices of the graph. A pendant (or leaf) vertex of $G$ is a vertex adjacent
to only  one
vertex of $G$. The unique vertex adjacent to a pendant vetrtex is called its support vertex. A graph $G$ is bipartite if the vertex set can be partitioned
into two non-empty subsets $X$ and $Y$ such that every edge of
$G$ has one
end vertex in $X$ and the other in $Y$. A bipartite graph in which each
vertex of $X$ is adjacent to  every vertex of $Y$ is called a
complete bipartite graph. If $|X|=p$ and $|Y|=q$, then
the complete bipartite graph is denoted by $K_{p,q}$.\par

The cartesian product of two graphs $G$ and $H$,
denoted by $G \square H$, is the graph with vertex set $V(G) \times
V(H)$ and any two vertices $(u_1, v_1)$ and $(u_2, v_2)$ are
adjacent in $G \square H$ if \emph{(i)} $u_1 = u_2$ and $v_1v_2
\in E(H)$, or \emph{(ii)} $u_1u_2 \in E(G)$ and $v_1 = v_2$. If $G=P_{m}$ and $H=P_{n}$, then the cartesian product $G \square H$ is called the $m\times n$ grid graph and is denoted by $G_{m,n}$. \par

The corona of two graphs $G_{1}=(V_{1},E_{1})$ and $G_{2}=(V_{2},E_{2})$, denoted by $G_{1}\odot G_{2}$, is the graph obtained by taking one copy of $G_{1}$ and $|V_{1}|$ copies of $G_{2}$, and then joining the $i^{th}$ vertex of $G_{1}$ to every vertex in the $i^{th}$ copy of $G_{2}$. \par

A rooted graph is a graph in which one vertex is labelled in a special way so as to distinguish it from other vertices. The special vertex is called the root of the graph. Let $G$ be a labelled graph on $n$ vertices. Let $H$ be a sequence of $n$ rooted graphs $H_{1}, H_{2},\ldots,H_{n}$. Then by $G(H)$ we denote the graph obtained by identifying the root of $H_{i}$ with the $i^{th}$ vertex of $G$. We call $G(H)$ the rooted product of $G$ by $H$ \cite{God}.\par

A Roman dominating function (RDF) on a graph $G=(V,E)$ is defined as a function $f:V\rightarrow \{0,1,2\}$ satisfying the condition that every vertex $v$ for which $f(v)=0$ is adjacent to at least one vertex $u$ for which $f(u)=2$. The weight of a RDF is the value $f(V)=\sum_{v\in V}f(v)$. The Roman domination number of a graph $G$, denoted by $\gamma_{R}(G)$, is the minimum weight of all possible RDFs on $G$.\par

Let $(V_{0}, V_{1}, V_{2}, V_{3})$ be the ordered partition of $V$ induced by a DRDF $f$, where $V_{i}=\{v\in V:f(v)=i\}$. Note that there exists a $1-1$ correspondence between the functions $f$ and the ordered partitions $(V_{0}, V_{1}, V_{2}, V_{3})$ of $V$. Thus we will write $f=(V_{0}, V_{1}, V_{2}, V_{3})$.\par

For any graph theoretic terminology and notations not mentioned here, the readers may refer to \cite{Bal}. The following propositions are useful in this paper.

\begin{prop}\label{1}
    \cite{Bee} In a double Roman dominating function of weight

    $\gamma_{dR}(G)$, no vertex needs to be assigned the value $1$.
\end{prop}

Hence, without loss of generality, in determining the value
$\gamma_{dR}(G)$  we can assume that $V_{1}=\phi$ for all double
Roman dominating functions under consideration.

\begin{prop}\label{2}
    \cite{Bee} For any graph $G$, $2\gamma(G)\leqslant \gamma_{dR}(G) \leqslant 3\gamma(G)$.
\end{prop}

\begin{prop}\label{3}
    \cite{Abd1} For $n\geqslant 3$,
    \[
    \gamma_{dR}(C_{n})=
    \begin{cases}
    n, & \text{if } n\equiv 0,\,2,3,4\,(mod \,6),\\
    n+1,& \text{if }n\equiv 1,\,5\,  (mod \,6).\\
        \end{cases}
    \]
\end{prop}

\begin{prop}\label{4}
    \cite{Bee} For any nontrivial connected graph $G$, $\gamma_{R}(G)< \gamma_{dR}(G) < 2\gamma_{R}(G)$.
\end{prop}

\section{Cartesian Product}
The Roman domination number of cartesian product graphs is studied in \cite{Yer}. As far as we know, there are no results on the double Roman domination  number of cartesian product graphs. In \cite{Bee}, it is proved that for every graph $G$, $\gamma_{R}(G)< \gamma_{dR}(G)$. Also it is proved in \cite{Wu} that $\gamma_{R}(G\square H)\geqslant \gamma(G) \gamma(H)$. Hence we can deduce a general relationship between the double Roman domination number of cartesian product graphs and the domination number of its factors as follows:
\[ \gamma_{dR}(G\square H)> \gamma(G)\gamma(H) \]

\begin{prop}
    Let $G$ be a graph. For any $\gamma_{dR}$-function $f=(V_{0}, V_{2}, V_{3})$ of $G$,\\
    (i) $|V_{3}|\leqslant \gamma_{dR}(G)-2\gamma(G)$ and\\
    (ii) $|V_{2}|\geqslant 3\gamma(G)-\gamma_{dR}(G)$.
\end{prop}

\begin{proof}
    Since $V_{2}\cup V_{3}$ is a dominating set for $G$, we have $\gamma(G)\leqslant |V_{2}|+|V_{3}|$. So, $2\gamma(G)\leqslant 2|V_{2}|+2|V_{3}|= \gamma_{dR}(G)-|V_{3}|$, and hence (i) is deduced. Also, $3\gamma(G)\leqslant 3|V_{2}|+3|V_{3}|=\gamma_{dR}(G)+|V_{2}|$, and hence (ii) is obtained.
\end{proof}

\begin{thm}\label{drdncartesian}
    For any graphs $G$ and $H$, $\gamma_{dR}(G\square H)\geqslant \frac{\gamma(G)\gamma_{dR}(H)}{2}$.

\end{thm}

\begin{proof}
    Let $V(G)$ and $V(H)$ be the vertex sets of $G$ and $H$ respectively. Let $f=(V_{0}, V_{2}, V_{3})$ be a $\gamma_{dR}$-function of $G\square H$. Let $S=\{u_{1}, u_{2},\ldots,u_{\gamma(G)}\}$ be a dominating set for $G$. Let $\{A_{1}, A_{2},\ldots,A_{\gamma(G)}\}$ be a vertex partition of $G$ such that $u_{i}\in A_{i}$ and $A_{i}\subseteq  N[u_{i}]$ (Note that this partition always exists and it may not be unique). Let $\{\Pi_{1}, \Pi_{2},\ldots,\Pi_{\gamma(G)}\}$ be the vertex partition of $G\square H$ such that $\Pi_{i}=A_{i}\times V(H)$ for every $i\in \{1, 2,\ldots,\gamma(G)\}$.\par

    For every $i\in \{1,2,\ldots,\gamma(G)\}$, let $f_{i}:V(H)\rightarrow \{0, 2, 3 \}$ be a function such that $f_{i}(v)=max\{f(u,v): u\in A_{i}\}$. For every $j\in \{0,2,3\}$, let $X_{j}^{(i)}=\{v\in V(H): f_{i}(v)=j\}$. Let $Y_{0}^{(i)}=\{x\in X_{0}^{(i)}: |N(x)\cap X_{2}^{(i)}|\leqslant 1$ and $ N(x)\cap X_{3}^{(i)}=\phi \}$. Hence, we have that $f_{i}'=(X_{0}^{(i)}-Y_{0}^{(i)}, X_{2}^{(i)}+Y_{0}^{(i)},X_{3}^{(i)})$ is a double Roman dominating function on $H$. Thus,
    \begin{eqnarray*}
        \gamma_{dR}(H) &\leqslant& 3|X_{3}^{(i)}|+2|X_{2}^{(i)}|+2|Y_{0}^{(i)}|\\
        &\leqslant& 3|V_{3}\cap \Pi_{i}|+2|V_{2}\cap \Pi_{i}|+2|Y_{0}^{(i)}|.
    \end{eqnarray*}
    Hence,
    \begin{eqnarray*}
        \gamma_{dR}(G\square H) &=& 3|V_{3}|+2|V_{2}|\\
        &=&\sum_{i=1}^{\gamma(G)}[3|V_{3}\cap \Pi_{i}|+2|V_{2}\cap \Pi_{i}|]\\
        &\geqslant& \sum_{i=1}^{\gamma(G)}[ \gamma_{dR}(H)-2|Y_{0}^{(i)}|]   \\
        &=&\gamma(G)\gamma_{dR}(H)-2\sum_{i=1}^{\gamma(G)}|Y_{0}^{(i)}|.
    \end{eqnarray*}

    So,
    \begin{equation}
        \sum_{i=1}^{\gamma(G)}|Y_{0}^{(i)}| \geqslant \frac{1}{2}[\gamma(G)\gamma_{dR}(H)-\gamma_{dR}(G\square H)].
    \end{equation}
    Now, for every $v\in V(H)$, let $Z^{v}\in \{0,1\}^{\gamma(G)}$ be a binary vector associated to $v$ as follows:

    \[Z_{i}^{v}=
    \begin{cases}
    1, & if\ \ v\in Y_{0}^{(i)},\\
    0,&  otherwise.
    \end{cases}
    \]
    Let $t_{v}$ be the number of components of $Z^{v}$ equal to one. Hence,
    \begin{equation}\sum_{v\in V(H)}t_{v}=\sum_{i=1}^{\gamma(G)}|Y_{0}^{(i)}|. \end{equation}
    Note that, if $Z_{i}^{v}=1$ and $u\in A_{i}$, then vertex $(u,v)$ belongs to $V_{0}$. Moreover $(u,v)$ is not adjacent to any vertex of $V_{3}\cap \Pi_{i}$ and is adjacent to at most one vertex of $V_{2}\cap \Pi_{i}$. So, since $V_{0}$ is double Roman dominated by $V_{2}\cup V_{3}$, there exists $u' \in X_{v}=\{x\in V(G): (x,v)\in V_{2}\cup V_{3}\}$ which is adjacent to $u$. Hence, $S_{v}=(S-\{u_{i}\in S:Z_{i}^{v}=1\})\cup X_{v}$   is a dominating set for $G$.\par
    Now, if $t_{v}>|X_{v}|$, then we have
    \begin{eqnarray*}
        |S_{v}| &\leqslant& |S|-t_{v}+|X_{v}|\\
        &=& \gamma(G)-t_{v}+|X_{v}| \\
        &<&  \gamma(G)-t_{v}+t_{v}=\gamma(G),
    \end{eqnarray*}
    which is  a contradiction. So, we have $t_{v}\leq |X_{v}|$ and we obtain
    \[\sum_{v\in V(H)}t_{v}\leq \sum_{v\in V(H)} |X_{v}|=|V_{2}\cup V_{3}|\]
    which leads to
    \begin{equation}
        2\sum_{v\in V(H)}t_{v}\leq 2|V_{2}|+2|V_{3}|\leq \gamma_{dR}(G\square H).
    \end{equation}

    Thus, by (1), (2) and (3), we deduce $\gamma_{dR}(G\square H)\geqslant \frac{\gamma(G)\gamma_{dR}(H)}{2}$.
\end{proof}

Proposition \ref{2} and Theorem \ref{drdncartesian} lead to the following result.

\begin{cor}
    For any graphs $G$ and $H$, $\gamma_{dR}(G\square H)\geqslant\frac{\gamma_{dR}(G)\gamma_{dR}(H)}{6}$.
\end{cor}

\begin{thm}
    For any graphs $G$ and $H$ of orders $n_{1}$ and $n_{2}$ respectively, $\gamma_{dR}(G\square H)\leqslant min\{n_{2
    }\gamma_{dR}(G),n_{1}\gamma_{dR}(H) \}$.
\end{thm}

\begin{proof}
    Let $f_{1}$ be a $\gamma_{dR}$-function of $G$. Let $f:V(G)\times V(H) \rightarrow \{0,2,3\}$ be a function defined by $f(u,v)=f_{1}(u)$. If there exists a vertex $(u,v)\in V(G)\times V(H)$ such that $f(u,v)=0$, then $f_{1}(u)=0$. Since $f_{1}$ is a $\gamma_{dR}$-function of $G$, there exists either $u_{1}\in N_{G}(u)$ such that $f_{1}(u_{1})=3$ or $u_{2},u_{3}\in N_{G}(u)$ such that $f_{1}(u_{2})=f_{1}(u_{3})=2$. Hence, we obtain that there exists either $(u_{1},v)\in N_{G \square H}((u,v))$ with $f((u_{1},v))=3$ or $(u_{2},v),(u_{3},v)\in N_{G \square H}((u,v))$ with $f((u_{2},v))=f((u_{3},v))=2$. So, $f$ is a DRDF of $G\square H$. Therefore,
    \begin{eqnarray*}
        \gamma_{dR}(G\square H) &\leqslant&  \sum_{(u,v)\in V(G)\times V(H)}f((u,v))\\
        &=& \sum_{v\in V(H)} \sum_{u\in V(G)}f_{1}(u)\\
        &=& \sum_{v\in V(H)}\gamma_{dR}(G)=n_{2}\gamma_{dR}(G).
    \end{eqnarray*}
    Similarly, we can prove that $\gamma_{dR}(G\square H)\leqslant n_{1}\gamma_{dR}(H)$ and hence the result is true.
\end{proof}

In \cite{Coc}, it is proved that for the $2\times n$ grid graph $G_{2,n}$, $\gamma_{R}(G_{2,n})=n+1$. Hence it is natural to study the double Roman domination number of grid graphs. For $n=2$, $G_{2,n}$ is $C_{4}$ and by proposition \ref{3},  $\gamma_{dR}(C_{4})=4$. So, in the next theorem, we omit the case when $n=2$.\\

\begin{thm}
    For the $2\times n$ grid graph $G_{2,n}, n\neq 2,\gamma_{dR}(G_{2,n})=\lfloor \frac{3n+4}{2}\rfloor$.
\end{thm}

\begin{proof}

    Let the vertices of $G_{2,n}$ be denoted by $(u_{1},v_{1}),\ldots,(u_{1},v_{n}),(u_{2},v_{1}),\ldots,(u_{2},v_{n})$ and define a DRDF $f$ as follows: If $n$ is odd,
    \[f(u_{i},v_{j})=
    \begin{cases}
    3, & for\ \ i=1\  and\  j=3+4k;\ i=2\ and\ j=1+4k\  for\   k\geqslant 0\  and\ j\leqslant n ,\\
    0,& otherwise.
    \end{cases}
    \]
    If $n$ is even,
    \[f(u_{i},v_{j})=
    \begin{cases}
    3, & for\ \ i=1\  and\  j=3+4k;\ i=2\ and\ j=1+4k\  for\   k\geqslant 0\  and\ j< n,\\
    2, & for \ i=1\  and\ j=n,\ if\ \ n\equiv 2\ (mod\ 4);\ i=2\  and\  j=n,\ if\ n\equiv 0\ (mod\ 4),\\
    0,& otherwise.
    \end{cases}
    \]
    It can be easily verified that $f$ is a DRDF and
    \[f(V)=
    \begin{cases}
    \frac{3(n+1)}{2}, &if \ n\ is\ odd,\\

    \frac{3n}{2}+2,&if \ n\ is\ even.
    \end{cases}
    \]
    i.e., $f(V)=\lfloor \frac{3n+4}{2}\rfloor$ and hence $\gamma_{dR}(G_{2,n})\leq \lfloor \frac{3n+4}{2}\rfloor$.\par

    \begin{figure}[h]
        \includegraphics[width=12cm]{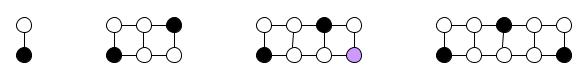}
        \caption{ \scriptsize {DRDF $f$ for $G_{2,n}, n=1,3,4,5$. Black circles denote vertices in $V_{3}$, grey circle denote vertex in $V_{2}$ and empty circles denote vertices in $V_{0}$.}}
        \label{fig:1}       % Give a unique label
    \end{figure}

    For the reverse inequality, let $\{x_{1}, x_{2},\ldots,x_{\gamma}\}$ be any dominating set for $G_{2,n}$. If $n$ is odd, $\{N[x_{1}],N[x_{2}],\ldots,N[x_{\gamma}]\}$ is a partition of vertex set of $G_{2,n}$ and $|N[x_{i}]|\geqslant 3$, for $i=1,2,\ldots,\gamma$. So we have to give $3$ to each $x_{i}$, $i=1,2,\ldots,\gamma$, under any DRDF and hence $\gamma_{dR}(G_{2,n})\geqslant \frac{3(n+1)}{2}$. (Note that $\gamma(G_{2,n})=\lceil\frac{n+1}{2}\rceil$). If $n$ is even, let $\{A_{1}, A_{2},\ldots,A_{\gamma}\}$ be any partition of vertex set of $G_{2,n}$ such that $x_{i}\in A_{i}$ and $A_{i}\subseteq N[x_{i}]$, $i=1,2,\ldots,\gamma$. Then $|A_{i}|=1$ for at most one $i$, say $k$, and $|A_{i}|\geq 3$, $i\neq k$. So we have to give $3$ to each $x_{i}$, $i=1,2,\ldots,\gamma$; $i\neq k$ and $2$ to $x_{k}$ under any DRDF. Hence $\gamma_{dR}(G_{2,n})\geqslant \frac{3n}{2}+2$, if $n$ is even. Hence the result follows.
\end{proof}

\section{Corona Operator}
In this section, first we find the double Roman domination number of $G\odot H$, where $H\ncong K_{1}$, and obtain bounds for $\gamma_{dR}(G\odot K_{1})$. We also prove that these bounds are strict and obtain a realization for every value in the range of the bounds obtained. The exact values of $\gamma_{dR}(G\odot K_{1})$ where $G$  is a path, a cycle, a complete graph or a complete bipartite graph  are also obtained. Also we prove that the value of $\gamma_{dR}((G\odot K_{1})\odot K_{1})$ depends only on the number of vertices in $G$.

\begin{prop}
    For every graph $G$ and every $H\ncong K_{1}$, $\gamma_{dR}(G\odot H)=3n$, where $n=|V(G)|$.
\end{prop}

\begin{proof}
    The function which assigns $3$ to all vertices of $G$ and $0$ to all other vertices is a DRDF of $G\odot H$ so that $\gamma_{dR}(G\odot H)\leqslant 3n$.  Also, there are $n$ mutually exclusive copies of $H$ each of which requires at least  weight $3$ in a DRDF. Hence the result is true.
\end{proof}

\begin{prop}\label{procor}
    For any graph $G$, $2n+1\leqslant\gamma_{dR}(G\odot K_{1})\leqslant3n$, where $n=|V(G)|$.
\end{prop}

\begin{proof}
    Let $V(G)=\{u_{1}, u_{2},\ldots,u_{n}\}$ and let $u_{i}'$ be the leaf neighbor of $u_{i}$ in $G\odot K_{1}$. We get a DRDF of $G\odot K_{1}$ by simply assigning the value $3$ to each $u_{i}'$ so that $\gamma_{dR}(G\odot K_{1})\leqslant3n$.\par
    To prove the left inequality, let $f$ be any DRDF of $G\odot K_{1}$. Being a pendant vertex, each $u_{i}'$ must be either in $V_{2}^{f}\cup V_{3}^{f}$ or adjacent to a vertex in $V_{3}^{f}$. Also, if $u_{i}'\in V_{2}^{f}$, for all $i=1,2,\ldots,n$, none of the vertices $u_{i}$ can be double Roman dominated by $u_{i}'$ alone. Therefore, $f(V)\geqslant 2n+1$ and hence $\gamma_{dR}(G\odot K_{1})\geqslant 2n+1$.
\end{proof}

\begin{prop}
    Any positive integer $a$ is realizable as the double Roman domination number of $G\odot K_{1}$ for some graph $G$ if and only if $2n+1 \leqslant a \leqslant 3n$, where $n=|V(G)|$.
\end{prop}

\begin{proof}
    Let $G$ be a graph with $|V(G)|=n$. If $\gamma_{dR}(G\odot K_{1})=a$, then by Proposition \ref{procor}, $2n+1\leqslant a\leqslant 3n$.\par

    To prove the converse part, take $G$ as $K_{1,m}\cup (n-m-1)K_{1}$. For definiteness, let $u_{1},u_{2},\ldots,u_{m+1}$ be the vertices of $K_{1,m}$ in which $u_{1}$ is the universal vertex and $u_{m+2},\ldots,u_{n}$ be the isolated vertices in $G$. Let $u_{i}'$ be the leaf neighbor of $u_{i}$ in $G\odot K_{1}$. Define $f$ on $V(G\odot K_{1})$ as follows:
    \[f(v)=
    \begin{cases}
    3, & if\ \ v=u_{i}'\ for \ i=1,m+2,m+3,\ldots,n,\\
    2,&  if\ \ v=u_{i}\ for\ i=2,3,\ldots,m+1,\\
    0,& otherwise.
    \end{cases}
    \]
    Clearly, $f$ is a $\gamma_{dR}$-function with weight $3(n-m)+2m=3n-m$. As $m$ varies from $0$ to $n-1$, we get $G$ with $\gamma_{dR}(G\odot K_{1})$ varies from $3n$ to $2n+1.$ (Note that $K_{1,0}$ is considered as $K_{1}$.) Hence, the result is true.
\end{proof}

\begin{prop}\label{prop Pn}

    \[\gamma_{dR}(P_{n}\odot K_{1})=
    \begin{cases}
    \frac{7n}{3}, & if\ \ n=3k,\\
    \frac{7n+2}{3}, & if\ \ n=3k+1,\\
    \frac{7n+1}{3}, & if\ \ n=3k+2.\\
    \end{cases}
    \]
\end{prop}

\begin{proof}
    Let $P_{n}: u_{1}u_{2}\ldots u_{n}$ be a path and let $u_{i}'$ be the vertex adjacent to $u_{i}$ in $P_{n}\odot K_{1}$. In a $\gamma_{dR}$-function, a pendant vertex must be either in $V_{2}$ or adjacent to a vertex in $V_{3}$. If $n=3k$ or $3k+2$, we have a  $\gamma_{dR}$-function of $P_{n}$ with $V_{2}=\phi$. If $n=3k+1$, let $f$ be a DRDF with minimal weight such that $V_{2}=\phi$. Define $g$ on $V(P_{n}\odot K_{1})$ as follows:
    \[g(v)=
    \begin{cases}
    3, & if\ \ v=u_{i}\in V_{3}^{f},\\
    2, & if\ \ v=u_{i}'\  with\  u_{i}\in V_{0}^{f},\\
    0 & otherwise.\\
    \end{cases}
    \]
    Clearly, $g$ is a $\gamma_{dR}$-function. Hence if $n=3k$, $\gamma_{dR}(P_{n}\odot K_{1})=3.\frac{n}{3}+2.\frac{2n}{3}=\frac{7n }{3}$. If $n=3k+1$, $\gamma_{dR}(P_{n}\odot K_{1})=3(\frac{n-1}{3}+1)+2\frac{2(n-1)}{3}=\frac{7n+2}{3}$. If $n=3k+2$, $\gamma_{dR}(P_{n}\odot K_{1})=3(\frac{n-2}{3}+1)+2(\frac{2(n-2)}{3}+1)=\frac{7n+1}{3}$.
\end{proof}

\begin{prop}

    \[\gamma_{dR}(C_{n}\odot K_{1})=
    \begin{cases}
    \frac{7n}{3}, & if\ \ n=3k,\\
    \frac{7n+2}{3}, & if\ \ n=3k+1,\\
    \frac{7n+1}{3}, & if\ \ n=3k+2.\\
    \end{cases}
    \]
\end{prop}

\begin{proof}
    The proof is similar to that of $P_{n}$.
\end{proof}

\begin{prop}

    $\gamma_{dR}(K_{n}\odot K_{1})=2n+1$.
\end{prop}

\begin{proof}
    Let $V(K_{n})=\{u_{1}, u_{2},\ldots,u_{n}\}$ and let $u_{i}'$ be the leaf neighbor of $u_{i}$ in $K_{n}\odot K_{1}$.    A $\gamma_{dR}$-function can be obtained for $K_{n}\odot K_{1}$ by assigning $3$ to any one vertex, say $u_{1}$ of $K_{n}$, $2$ to  $u_{i}'$ with $i\neq 1$, and $0$ to all other vertices. Hence   $\gamma_{dR}(K_{n}\odot K_{1})=2n+1$.
\end{proof}

\begin{prop}
    \[\gamma_{dR}(K_{p,q}\odot K_{1})=
    \begin{cases}
    2(p+q)+1, & if\ \ p=1\ or\ q=1,\\
    2(p+q+1),&  otherwise.
    \end{cases}
    \]
\end{prop}

\begin{proof}
    Let $V(K_{p,q})=\{u_{1},u_{2},\ldots,u_{p},v_{1},v_{2},\ldots,v_{q}\}$ and let $u_{i}'$ be the leaf neighbor of $u_{i}$, for $i=1,2,\ldots,p$ and $v_{j}'$ be the vertex adjacent to $v_{j}$, for $j=1,2,\ldots,q$ in $K_{p,q}\odot K_{1}$. By the left inequality of Proposition \ref{procor}, $\gamma_{dR}(K_{p,q}\odot K_{1})\geqslant 2(p+q)+1$.\\
    \textbf{    \textit{Case 1 :}} $p=1$ or $q=1$.\\
    For definiteness, let $p=1$. Then the function $f$ defined by
    \[f(u)=
    \begin{cases}
    3, & for \ u=u_{1},\\
    2, & for\  all\  u=u_{i}', i=2,3,\ldots,p\ and \  u=v_{j}', j=1,2,\ldots,q,\\
    0, & otherwise,
    \end{cases}
    \]
    is a DRDF of $K_{p,q}\odot K_{1}$ with weight $2(p+q)+1$. Therefore, $\gamma_{dR}(K_{p,q}\odot K_{1})=2(p+q)+1$.\\
    \textbf{\textit{Case 2 :}} $p,q \geqslant 2$.   \\
    Define $f$ as follows:
    \[f(u)=
    \begin{cases}
    3, & for \ u=u_{1}\ and \ u=v_{1},\\
    2, & for\  all\  u=u_{i}',\  i=2,3,\ldots,p\ and \  u=v_{j}',\  j=2,3,\ldots,q,\\
    0, & otherwise.
    \end{cases}
    \]
    $f$ is a DRDF of $K_{p,q}\odot K_{1}$ with weight $6+2(p+q-2)=2(p+q+1)$ and hence $\gamma_{dR}(K_{p,q}\odot K_{1})\leqslant 2(p+q+1)$. For the reverse inequality, if possible suppose that there exists a DRDF $g$ of $K_{p,q}\odot K_{1}$ with weight $2(p+q)+1$. Out of $p+q$ pendant vertices in $K_{p,q}\odot K_{1}$, let $k$ vertices be in $V_{2}^{g}$. Then the remaining $p+q-k$ pendant vertices  are either in $V_{3}^{g}$ or adjacent to vertices in $V_{3}^{g}$. Hence the weight of $g$, $g(V)=2(p+q)+1\geqslant 2k+3(p+q-k)$ which implies $k\geqslant p+q-1$. If $k> p+q-1$, then $k=p+q$ so that all the pendant vertices are in $V_{2}^{g}$ and none of them can double Roman dominate any of the non pendant vertices. Therefore, we need more vertices having non zero values under $g$ which contradicts the fact that $g(V)=2(p+q+1)$. If $k=p+q-1$,
    then one pendant vertex, say $x$, is either in $V_{3}^{g}$ or adjacent to a vertex in $V_{3}^{g}$. If $x$ is in $V_{3}^{g}$, then $x$ can double Roman dominate only its support vertex. If $x$ is adjacent to a vertex in $V_{3}^{g}$, then its support vertex, say $y$, is in $V_{3}^{g}$ and $y$ cannot double Roman dominate any of the remaining vertices in the partite set of $K_{p,q}$ containing $y$. In either case, we need more vertices having non zero values under $g$, which again leads to a contradiction as above. Hence the result is true.
    \end{proof}

\begin{prop}
    For any graph $G$, $\gamma_{dR}((G\odot K_{1})\odot K_{1})=5n$, where $n=|V(G)|$.
\end{prop}

\begin{proof}
    Let $G$ be a graph with vertex set $V(G)=\{u_{1},u_{2},\ldots,u_{n}\} $ and let $v_{i}$ be the leaf neighbor of  $u_{i}$ in $G\odot K_{1} $. Let $u_{i}'$ and $v_{i}'$ be the leaf neighbors of  $u_{i}$ and  $v_{i}$ respectively in $(G\odot K_{1})\odot K_{1}$. Then $(G\odot K_{1})\odot K_{1}$ contains $n$ vertex disjoint $P_{4}$'s, $u_{i}'u_{i}v_{i}v_{i}'$, for $i= 1,2,\ldots,n$. Let $f$ be any $DRDF$ on $(G\odot K_{1})\odot K_{1}$. Then the two pendant vertices, $u_{i}'$ and $v_{i}'$, in each $P_{4}$ should be either in $V_{2}^{f}\cup V_{3}^{f}$ or adjacent to a vertex in $V_{3}^{f}$. If all the pendant vertices are in $V_{2}^{f}$, then to double Roman dominate non pendant vertices, $u_{i}$ and $v_{i}$, we need more vertices with non zero values in each $P_{4}$. Also note that pendant vertices have no common neighbors. Hence, under $f$, the sum of the values of vertices in each of the above mentioned $P_{4}$'s must be at least $5$. Therefore, $f(V)\geqslant 5n$.\par

    To prove the reverse inequality, define $g$ as follows:
    \[g(u)=
    \begin{cases}
    3, & for \ u=u_{i},\\
    2, & for\  \  u=v_{i}',\\
    0, & otherwise.
    \end{cases}
    \]
    Clearly $g$ is a DRDF on $(G\odot K_{1})\odot K_{1}$ with $g(V)=5n$. Hence, the result is true.
\end{proof}

\noindent \textbf{Note: }
\begin{enumerate}
    \item   Corona of any graph $G$ with $K_{1}$ can be thought of as rooted product of $G$ by $H$ where $H$ is a sequence of $n$ rooted $P_{2}$'s.
    \item Corona of $P_{n}$ with $K_{1}$ is called comb. \end{enumerate}

\section{ Addition of Twins}

In this section, we study the impact of addition of twins on double Roman domination number.

\begin{thm}
    Let $G$ be a graph and $u\in V(G)$. Let $H$ be the graph obtained from $G$ by attaching a true twin $u'$ to $u$. Then $\gamma_{dR}(G)\leqslant \gamma_{dR}(H)\leqslant \gamma_{dR}(G)+1$.
\end{thm}

\begin{proof}
    For the left inequality, let $f$ be a $ \gamma_{dR}$-function of $H$. Then $f|_{V(G)}$ will be a DRDF on $G$ in either of the
    following cases.\\
    (i)$f(u)=3$\\
    (ii)$f(u)=2$ and $f(u')=0$.\\
    The case when $f(u)=2$ and $f(u')=3$ won't arise. If $f(u)=f(u')=2$, then reassign the value of $u$ as $3$. If $f(u)=0$, then interchange the values of $u$ and $u'$. In these two cases, the restriction of reassignment to $V(G)$ will be a DRDF of $G$.
    Hence $\gamma_{dR}(G)\leqslant \gamma_{dR}(H)$.\par
    To prove the right inequality, let $f$ be a $ \gamma_{dR}$-function of $G$. If $u\in V_{0}^{f}\cup  V_{3}^{f}$ or $|N(u)\cap  V_{3}^{f}|\geqslant 1$ or $|N(u)\cap  V_{2}^{f}|\geqslant 2$, then the function $f'$ which assigns $0$ to $u'$ and $f(u)$ to every other vertices is a DRDF of $H$. Otherwise define $f'$ on $V(H)$ as follows.
    \[f'(v)=
    \begin{cases}
    3, & if\ \ v=u ,\\
    0,&  if\ \ v=u',\\
    f(v),& otherwise.
    \end{cases}
    \]
    Then $f'$ is a DRDF of $H$ with weight one greater than that of $f$ which completes the proof.
\end{proof}

\begin{thm}
    Let $G$ be a graph and $u\in V(G)$. Let $H$ be the graph obtained from $G$ by attaching a false twin $u'$ to $u$. Then $\gamma_{dR}(G)\leqslant \gamma_{dR}(H)\leqslant \gamma_{dR}(G)+2$.
\end{thm}

\begin{proof}
    The proof of the left inequality is  same as that of true twin.\par
    To prove the right inequality, let $f$ be a $ \gamma_{dR}$-function of $G$. If  $|N(u)\cap  V_{3}^{f}|\geqslant 1$ or $|N(u)\cap  V_{2}^{f}|\geqslant 2$, then the function $f'$ which assigns $0$ to $u'$ and $f(u)$ to every other vertices is a DRDF of $H$ with weight equal to $\gamma_{dR}(G)$. Otherwise the function $g$ which assigns $2$ to $u'$ and $f(u)$ to all other vertices is a DRDF of $H$ with weight $2$ greater than that of $f$ which completes the proof.
\end{proof}

\section{Double Roman Domination and Roman Domination}
By Proposition \ref{4}, for any nontrivial connected graph, $ \gamma_{R}(G)< \gamma_{dR}(G)< 2\gamma_{R}(G)$ \cite{Bee}. Then it is natural to investigate whether an ordered pair $(a,b)$ is realizable as the Roman domination number and double Roman domination number of some non trivial connected graph if $a<b<2a$. In \cite{Bee}, R. A. Beeler et.al gave an example for which $(a,2a-1)$ is realizable. The ordered pair $(a, a+\lceil \frac{a}{2}\rceil)$ is realizable for path $P_{a+\lceil \frac{a}{2}\rceil-1}$ (\cite{Anu},\cite{Bee},\cite{Xu}). The ordered pair $(2,3)$ is realization for $P_{2}$. So we restrict $a$ to be greater than $2$ in the next theorem.

\begin{thm}\label{a,a+1}
    The ordered pair $(a,a+1)$, where $a$ is an even number greater than $2$, is not realizable as the Roman domination number and  the double Roman domination number of any graph.
\end{thm}

\begin{proof}
    Let $G$ be any  graph with $\gamma_{dR}(G)=a+1$. Let $f=(V_{0},V_{2}, V_{3})$ be a $\gamma_{dR}$-function of $G$. Since $a+1$ is odd, $|V_{3}|\geqslant 1$. If $V_{2}=\phi$, then $|V_{3}|> 2$. Define a function $g=(V_{0}', V_{1}', V_{2}')$ as follows: $V_{0}'=V_{0}, V_{1}'=\phi, V_{2}'=V_{3}$. Clearly, under $g$, every vertex in $V_{0}'$ has a neighbor in $V_{2}'$. Hence $g$ is a RDF so that $\gamma_{R}(G)\leq a-2$. If $V_{2}\neq \phi$, then we can find a vertex $u\in V_{2}$ such that $g'=(V_{0}'', V_{1}'', V_{2}'')$ where $V_{0}''=V_{0}, V_{1}''={u}, V_{2}''=V_{2}-\{u\}\cup V_{3}$ is a RDF of $G$ so that  $\gamma_{R}(G)\leq a-1$. Hence the result is true.
\end{proof}

\noindent \textbf{Observation:} The ordered pair $(2,4)$ is not realizable as the Roman domination number and  the double Roman domination number of any graph.
\begin{proof}
    Let $G$ be a graph with $\gamma_{dR}(G)=4$. Then $G$ is a graph of order $n\geqslant 4$ and there exists a pair of independent twin vertices $u_{1}$, $u_{2}$, each of which dominates all other vertices. So in any RDF the total weight will be at least $3$ and hence $(2,4)$ is not realizable. \end{proof}

The only graph with $\gamma_{R}=1$ is $K_{1}$ for which $\gamma_{dR}=2$. Also, $(2,3)$ is realizabe for any graph with a universal vertex. Our final result shows that every value in the range of Proposition \ref{4}, except for  those values which have been already mentioned, is realizable for connected bipartite graph.
\begin{thm}
    The ordered pair $(a,b)$, where $a>1$ and $b\neq 3$ is an integer greater than $a$, is realizable as the Roman domination number and the double Roman domination number of some connected bipartite graph if $a=\lfloor\frac{b}{2}\rfloor+1, \lfloor\frac{b}{2}\rfloor+2,\ldots,2\lfloor\frac{b}{2}\rfloor-1 $.
\end{thm}

\begin{proof}
    The proof is by construction. To obtain a graph with $\gamma_{R}=\lfloor\frac{b}{2}\rfloor+i$ and $\gamma_{dR}=b$, for $i=1,2,\ldots,\lfloor\frac{b}{2}\rfloor-1$, construct a bipartite graph $G_{i}(X,Y)$ as follows: Take $\lfloor\frac{b}{2}\rfloor$ vertices $x_{1},x_{2},\ldots,x_{\lfloor\frac{b}{2}\rfloor}$ in $X$. Corresponding to each pair of vertices, $ (x_{j},x_{k})$, for $1\leqslant j\leqslant i$, $j< k \leqslant \lfloor\frac{b}{2}\rfloor$, take two vertices in $Y$ and make them adjacent with $x_{j}$ and $x_{k}$.\\
    \noindent \textbf{\emph{Case 1:}} $b$ is even.\\
    A  DRDF of $G_{i}(X,Y)$ can be obtained by giving $2$ to each and every vertex in $X$ and $0$ to all vertices in $Y$ so that $\gamma_{dR}(G_{i}(X,Y))\leqslant b$.\\
    \textbf{\emph{Case 2:}} $b$ is odd.\\
    In addition to the above construction, take two pendant vertices in $Y$ and make them adjacent with $x_{1}$. In this case a DRDF can be obtained by giving $3$ to $x_{1}$, $2$ to remaining vertices in $X$ and $0$ to all vertices in $Y$  so that $\gamma_{dR}(G_{i}(X,Y))\leqslant b$.\par
    In both cases, a RDF can be obtained by assigning $2$ to $x_{j}$, for $1\leqslant j\leqslant i$, $1$ to $x_{j}$, for $i+1\leqslant j\leqslant \lfloor\frac{b}{2}\rfloor $, and $0$ to all vertices in $Y$ so that $\gamma_{R}(G_{i}(X,Y))\leqslant \lfloor\frac{b}{2}\rfloor+i$. It is easy to verify that these functions are in fact  minimum. Hence the result is true.\end{proof}




\ni \textbf{Acknowledgements:}
    The first author
    thanks University Grants Commission for granting fellowship under
    Faculty Development Programme (F.No.FIP/$12^{th}$ plan/KLMG 045 TF
    07 of UGC-SWRO).



\end{document}